\documentclass[12pt,reqno]{amsart}
\usepackage{amscd,amsmath,amsthm,amssymb}
\usepackage{color}
\usepackage{tikz}
\usepackage{url}
\usepackage{circuitikz}
\usepackage{float}

\newtheorem{Theorem}{Theorem}[section]
\newtheorem{Lemma}[Theorem]{Lemma}
\newtheorem{Corollary}[Theorem]{Corollary}
\newtheorem{Proposition}[Theorem]{Proposition}

\newtheorem{Question}[Theorem]{Question}

\newtheorem{Conj}[Theorem]{Conjecture}

\def\qed{\ifhmode\textqed\fi
	\ifmmode\ifinner\quad\qedsymbol\else\dispqed\fi\fi}
\def\textqed{\unskip\nobreak\penalty50
	\hskip2em\hbox{}\nobreak\hfill\qedsymbol
	\parfillskip=0pt \finalhyphendemerits=0}
\def\dispqed{\rlap{\qquad\qedsymbol}}

\def\Ker{\textup{Ker}}

\def\m{\mathfrak{m}}

\def\Ass{\textup{Ass}}

\def\depth{\textup{depth\,}}
\def\lcm{\textup{lcm}}

\def\reg{\textup{reg}}

\def\lex{\textup{lex}}

\def\ini{\textup{in}}

\def\rank{\textup{rank}}
\def\set{\textup{set}}
\def\dstab{\textup{dstab}}
\def\mat{\textup{mat}}

\begin{document}
	
	\title{Rees algebras of complementary edge ideals}
	\author{Antonino Ficarra, Somayeh Moradi}
	
	\address{Antonino Ficarra, BCAM -- Basque Center for Applied Mathematics, Mazarredo 14, 48009 Bilbao, Basque Country -- Spain, Ikerbasque, Basque Foundation for Science, Plaza Euskadi 5, 48009 Bilbao, Basque Country -- Spain}
	\email{aficarra@bcamath.org,\,\,\,\,\,\,\,\,\,\,\,\,\,antficarra@unime.it}
	
	\address{Somayeh Moradi, Department of Mathematics, Faculty of Science, Ilam University, P.O.Box 69315-516, Ilam, Iran}
	\email{so.moradi@ilam.ac.ir}
	
	
	\begin{abstract}
		In this paper we investigate the Rees algebras of squarefree monomial ideals $I \subset S=K[x_1,\dots,x_n]$ generated in degree $n-2$, where $K$ is a field. Every such ideal arises as the complementary edge ideal $I_c(G)$ of a finite simple graph $G$. We describe the defining equations of the Rees algebra $\mathcal{R}(I_c(G))$ in terms of the combinatorics of $G$.  If $G$ is a tree or a unicyclic graph whose unique induced cycle has length $3$ or $4$, we prove that $\mathcal{R}(I_c(G))$ is Koszul.   We also determine the asymptotic depth of the powers of $I_c(G)$, proving that $\lim_{k \to \infty}\text{depth}\, S/I_c(G)^k=b(G)$, where $b(G)$ is the number of bipartite connected components of $G$. Finally, we show that the index of depth stability of $I_c(G)$ is at most $n-2$, and equality holds when $G$ is a path graph.
	\end{abstract}
	
	\maketitle\vspace*{-2.1em}
	\section*{Introduction}
	
	Let $I$ be a squarefree monomial ideal generated in degree $d$ in the polynomial ring $S=K[x_1,\dots,x_n]$ over a field $K$. A famous theorem of Herzog, Hibi and Zheng \cite{HHZ} (see, also, \cite{FShort}) guarantees that if $I$ has a $2$-linear resolution, then $I^k$ has a $2k$-linear resolution for all $k\ge1$. Examples of Terai and Sturmfels show that in general this property does not hold in degree $d=3$. In \cite{FM}, we investigated for which degrees $d$ an analogue of the Herzog-Hibi-Zheng theorem holds, and it turned out that this question has a positive answer precisely for $d\in\{0,1,2,n-2,n-1,n\}$. Besides the case $d=2$ already addressed in \cite{HHZ}, and the cases $d\in\{0,1,n-1,n\}$ which are trivial, the case $d=n-2$ stands out. When $d=n-2$, each minimal monomial generator of $I$ is of the form $(x_1\cdots x_n)/(x_ix_j)$ for some $i\ne j$. This observation naturally leads to the concept of \textit{complementary edge ideal} \cite{FM1}, introduced independently in \cite{HQS}.
	
	Let $G$ be a finite simple graph on the vertex set $V(G)=[n]=\{1,2,\dots,n\}$ and with the edge set $E(G)$. The \textit{complementary edge ideal} of $G$ is defined as
	$$
	I_c(G)\ =\ ((x_1\cdots x_n)/(x_ix_j):\ \{i,j\}\in E(G)).
	$$
	
	Any squarefree monomial ideal $I\subset S$ generated in degree $n-2$ is the complementary edge ideal of some graph $G$ on the vertex set $V(G)=[n]$. More generally, the concept of complementary ideal of a squarefree monomial ideal was first considered by Villarreal in \cite{RV2}, 
    and later was extended for arbitrary monomial ideals in \cite{ALS}. 
	
	Let $c(G)$ be the number of connected components of $G$ having at least two vertices. In \cite[Theorem B]{FM}, we proved that $I_c(G)$ has linear resolution, if and only if, $I_c(G)^k$ has a linear resolution for all $k\ge1$, if and only if, $c(G)=1$. To establish this result, we briefly investigated the structure of the Rees algebra of $I_c(G)$,
	$$
	\mathcal{R}(I_c(G))\ =\ \bigoplus_{k\ge0}I_c(G)^kt^k.
	$$
	
	Our goal in this paper is to systematically study the Rees algebra of a complementary edge ideal $I_c(G)$ in terms of the combinatorics of $G$.   \smallskip
	
	In Section \ref{sec1}, we describe the defining equations of the Rees algebra of $I_c(G)$ in terms of the even closed walks of another graph $G^*$. The graph $G^*$ is obtained from the graph $G$ on vertex set $V(G)=[n]$ by adjoining a new vertex $n+1$ and connecting it to all vertices of $G$. In Theorem \ref{Thm:degx}, we prove that the $x$-degree of any primitive binomial relation of $\mathcal{R}(I_c(G))$ is at most 2. Using this result and the computation of the function $k\mapsto\reg\,I_c(G)^k$ accomplished in \cite[Theorem 4.1]{FM1}, in Theorem \ref{Thm:x-reg-IcG} we prove that the $x$-regularity of $\mathcal{R}(I_c(G))$ satisfies the inequalities
	$$
	c(G)-1\ \le\ \reg_x\mathcal{R}(I_c(G))\ \le\ |V(G)|-1.
	$$
	Moreover, in Corollary \ref{bound-reg} we prove the inequality $\reg\,\mathcal{R}(I_c(G))\le|V(G)|$ for trees and connected unicyclic graphs with unique cycle of length either $3$ or $4$. Whether this inequality holds in general is an open question at the moment (Question \ref{Quest:reg-R-I-c-G}).\smallskip
	
	In Section \ref{sec2}, we consider the problem of characterizing when $\mathcal{R}(I_c(G))$ has a quadratic Gr\"obner basis and when $\mathcal{R}(I_c(G))$ is a Koszul algebra. The latter problem appears to be very difficult. For instance if $G$ is a complete graph and we remove from $G$ just one edge, then $\mathcal{R}(I_c(G))$ is Koszul. In Theorem \ref{Koszul} we give necessary conditions for the Koszulness of $\mathcal{R}(I_c(G))$. We prove that $\mathcal{R}(I_c(G))$ has a quadratic Gr\"obner basis, and hence is a Koszul ring, if $G$ is a tree (Theorem \ref{tree}) or a connected unicyclic graph whose unique induced cycle has length 3 or 4 (Theorem \ref{unicylic}(b)).  \smallskip
	
	In Section \ref{sec3}, collecting results of Villarreal \cite{RV}, Hibi and Ohsugi \cite{HO}, and Ansaldi, Lin and Shen \cite{ALS}, in Theorem \ref{Thm:Fiber-IcG} we see that the Rees algebra of the edge ideal $I(G)$ is normal, if and only if, $\mathcal{R}(I_c(G))$ is normal, if and only if, $G$ satisfies the odd cycle condition. Let $b(G)$ be the number of bipartite connected components of $G$. Here, we regard an isolated vertex of $G$ as a bipartite connected component of $G$. Combining Theorem \ref{Thm:Fiber-IcG}, \cite[Theorem 3.1]{ALS} and \cite[Lemma 10.2.6]{RV}, it follows immediately that the analytic spread $\ell(I_c(G))$ of $I_c(G)$ is $|V(G)|-b(G)$. We prove this directly and independently using linear algebra.\smallskip
	
	By Brodmann \cite{Br}, the limit $\lim_{k\rightarrow\infty}\depth S/I^k$ exists for any ideal $I\subset S$. The least integer $k_0>0$ such that $\depth\,S/I^k=\depth\,S/I^{k_0}$ for all $k\ge k_0$ is called the \textit{index of depth stability} of $I$ and is denoted by $\dstab\,I$. By \cite[Proposition 10.3.2]{HHBook} and Corollary \ref{Cor:ell-I(G)}, we have $\lim_{k\rightarrow\infty}\depth S/I_c(G)^k\le|V(G)|-\ell(I_c(G))=b(G)$ and equality holds if $\mathcal{R}(I_c(G))$ is Cohen-Macaulay. Surprisingly, we prove in Theorem \ref{Thm:limitdepth-Ic(G)} that $\lim_{k\rightarrow\infty}\depth S/I_c(G)^k=b(G)$ for any graph $G$, and $\dstab\,I_c(G)\le|V(G)|-2$. In Proposition \ref{Prop:Pn}, we prove that this bound for the index of depth stability of $I_c(G)$ is sharp. The precise values of the depth function $k\mapsto\depth\,S/I_c(G)^k$ remain unknown for $1\le k<|V(G)|-2$. It would be also nice to have a precise formula for $\dstab\,I_c(G)$. In the case that $G$ is a tree, experimental evidence suggests that $\dstab\,I_c(G)$ is the length of the longest induced path of $G$ minus two.\smallskip
	
	In view of the results in this paper, and several experimental evidence, we expect that $\mathcal{R}(I_c(G))$ is a Cohen-Macaulay ring for any graph $G$ (Conjecture \ref{Conj:R(Ic(G))-CM}).
	
	\section{The defining equations of $\mathcal{R}(I_c(G))$}\label{sec1}
	
	In this section we study the defining ideal of the Rees algebra $\mathcal{R}(I_c(G))$. In \cite{V}, the defining ideal of $\mathcal{R}(I(G))$ is described in terms of the syzygies of $I$ and the defining ideal of the edge ring $K[G]$. To this aim, the concept of an even closed walk in the graph $G$ played a crucial role. Since defining equations  of $\mathcal{R}(I(G))$ and $\mathcal{R}(I_c(G))$ are closely related, below we will use the correspondence between even closed walks and the defining equations of $\mathcal{R}(I(G))$  to study the defining equations of $\mathcal{R}(I_c(G))$.
    
    We fix the following notation, which we will use throughout this and the next section. For a monomial ideal $I\subset S$, let $\mathcal{G}(I)$ be the minimal monomial generating set of $I$. Given $A\subset[n]=\{1,\dots,n\}$, we put ${\bf x}_A=\prod_{i\in A}x_i$, and we set ${\bf x}_\emptyset=1$.\smallskip
	
	Let $G$ be a finite simple graph with $V(G)=[n]$ and $E(G)=\{e_1,\ldots,e_m\}$. For all $i=1,\dots,m$, we set $u_i={\bf x}_{[n]}/{\bf x}_{e_i}$. Then $\mathcal{G}(I_c(G))=\{u_1,\dots,u_m\}$. Set $I=I_c(G)$. Let $T=S[y_1,\dots,y_m]$ be a polynomial ring and let $\varphi:T\rightarrow\mathcal{R}(I)$ be the $S$-algebra homomorphism defined by $\varphi(y_i)=u_it$ for $i=1,\dots,m$. We set $J=\Ker\,\varphi$. 
	
    Moreover, let $I(G)=(x_ix_j:\{i,j\}\in E(G))$ be the edge ideal of $G$, let $T'=S[z_1,\dots,z_m]$ be a polynomial ring, let  $\varphi':T'\rightarrow\mathcal{R}(I(G))$ be the $S$-algebra homomorphism defined by $\varphi'(z_i)={\bf x}_{e_i}t$ for all $i=1,\dots,m$, and let $J'=\Ker\,\varphi'$. It is easily seen that any binomial relation $h=uy_{i_1}\cdots y_{i_k}-vy_{j_1}\cdots y_{j_k}\in J$ corresponds to a binomial relation $h'=uz_{j_1}\cdots z_{j_k}-vz_{i_1}\cdots z_{i_k}\in J'$, and vice versa.
	The Rees algebra $\mathcal{R}(I(G))$ is isomorphic to the edge ring $$K[G^*]=K[{\bf x}_e:\, e\in E(G^*)],$$ where $G^*$ is the graph obtained from $G$ by adding a new vertex $n+1$ to $G$ and connecting it to all  vertices of $G$ . So the relation $h'$ and hence $h$ corresponds to an even closed walk in $G^*$. Moreover, if $h$ belongs to a reduced Gröbner basis of $J$, then $h$ and hence $h'$ are primitive binomials, which implies that the corresponding even closed walk in $G^*$ is primitive, see \cite[Corollary 10.1.5]{HHBook}, \cite[Proposition 3.1]{V} and \cite[Lemma 10.1.9]{RV}. Any fiber relation $h=y_1^{a_1}\cdots y_m^{a_m}-y_1^{b_1}\cdots y_m^{b_m}$ in a reduced Gröbner basis of $J$ comes from a fiber relation $h'$ of $J'$. By \cite[Lemma 10.1.9]{RV}, $h$ and hence $h'$ is a primitive 
	binomial. So by \cite[Proposition 10.1.8]{RV}, we have $a_i\leq 2$ and $b_i\leq 2$ for all $i$. In Theorem \ref{Thm:degx}, we give some bound for the $x$-degree of binomials in a reduced Gröbner basis of $J$.   
    
	Since an isolated vertex of $G$ is of degree one in $G^*$, it does not belong to an even closed walk in $G^*$. Hence, removing isolated vertices from a graph $G$ does not change the ideals $J'$ and $J$. So in order to study the defining ideal $J$, we may assume that $G$ has no isolated vertices.
	
	\begin{Theorem}\label{Thm:degx}
		Let $G$ be a finite simple graph. Then there is a monomial order $<$ on $T$ such that a reduced Gröbner basis of $J$ with respect to $<$ consists of binomials of the form $f=uy_{i_1}\cdots y_{i_k}-vy_{j_1}\cdots y_{j_k}$, where $u$ and $v$ are monomials in $S$ of degree at most $2$
	\end{Theorem}
	\begin{proof}
    By the discussion prior to the theorem, we may assume that $G$ has no isolated vertices. Let $G_1,\ldots,G_r$ be the connected components of $G$. We fix a labeling on $V(G)$ as follows. Let $n_i=|V(G_i)|$ for all $i$. We label $V(G_1)$ by $1,\ldots,n_1$ such that $G_1\setminus\{1,\ldots,s\}$ is connected for all $s<n_1$. Such a labeling exists, as $G_1$ is connected (see the proof of \cite[Theorem 3.1(a)]{FM}). Suppose that $V(G_{i-1})$ is labeled. Next, we label $V(G_i)$ by $(n_1+\dots+n_{i-1}+1)$, $\ldots$, $(n_1+\dots+n_{i-1}+n_i)$ such that $G_i\setminus \{(n_1+\dots+n_{i-1}+1),\ldots, (n_1+\dots+n_{i-1}+s)\}$ is connected for all $s<n_i$. 
		
		Fix the lexicographic order $<$ on $T$ induced by $x_1>\dots>x_n>y_1>\dots>y_m$.
		Consider a minimal monomial generator $uy_{i_1}\cdots y_{i_k}\in \ini_<(J)$, where $u\in S$ is a monomial. We show that $\deg(u)\leq 2$. If $r=1$, then $G$ is a connected graph. Hence, by the proof of \cite[Theorem 3.1]{FM1}, we conclude that $\deg(u)\leq 1$. So in this case we are done. Now, assume that $r\geq 2$. Suppose that $\deg(u)\geq 2$. We prove that $\deg(u)=2$. 
		There exists a binomial $h=uy_{i_1}\cdots y_{i_k}-vy_{j_1}\cdots y_{j_k}\in J$ with $\ini_<(h)=uy_{i_1}\cdots y_{i_k}$, such that $v\in S$ is a monomial. By \cite[Theorem 3.13]{HHBook2} we may assume that $h$ is a primitive binomial. Then $h'=uz_{j_1}\cdots z_{j_k}-vz_{i_1}\cdots z_{i_k}\in J'$ is a primitive binomial, which corresponds to a primitive even closed walk in $G^*$, say $W$. Since $\deg(u)\geq 2$, $W$ passes the vertex $n+1$ at least two times.
		We may write $W$ as $$n+1,q_1,q_2,\ldots,q_d,n+1,q_{d+1},\ldots,q_{d+s},n+1,\ldots$$ where $q_i$'s are vertices in $G$. 
		Since $W$ is primitive, $d$ and $s$ are even numbers. Otherwise $W$ has a proper even closed subwalk, which contradicts to $W$ being primitive. So we obtain that $W': n+1,q_1,q_2,\ldots,q_d,n+1,q_{d+1},\ldots,q_{d+s},n+1$ is an even closed subwalk of $W$. Since $W$ is primitive, this implies that $W=W'$. Therefore, $W$ passes the vertex $n+1$ precisely two times. Hence, $\deg(u)=\deg(v)=2$.    
	\end{proof}
	
	As a consequence, we then have

	\begin{Theorem}\label{Thm:x-reg-IcG}
		Let $G$ be a finite simple graph on $n$ vertices. Then
		$$
		c(G)-1\ \le\ \reg_{x}\mathcal{R}(I_c(G))\ \le\ n-1.
		$$
	\end{Theorem}
	\begin{proof}
		Since $I_c(G)$ is equigenerated in degree $n-2$, \cite[Theorem 1.1]{HHZ} implies that
		$$
		\reg\,I_c(G)^k\ \le\ (n-2)k+\reg_{x}\mathcal{R}(I_c(G)),
		$$
		for all $k\ge1$. On the other hand, by \cite[Theorem 4.1]{FM1}, we have
		$$
		\reg\,I_c(G)^k\ =\ (n-2)k+c(G)-1,
		$$
		for all $k\gg0$. Hence $\reg_{x}\mathcal{R}(I_c(G))\ge c(G)-1$.
		
		To prove the upper bound, we will use Theorem~\ref{Thm:degx}   and the Taylor resolution of $\ini_<(J)$. By Theorem~\ref{Thm:degx} we see that each multigraded shift in the $i$th homological degree of  the Taylor resolution of $\ini_<(J)$ has $x$-degree at most $2i$. Hence, using upper semi-continuity (see \cite[Theorem 3.3.4(c)]{HHBook}), $$ \reg_{x}\mathcal{R}(I_c(G))=\reg_x\,T/J\leq \reg_x T/\ini_<(J)\le \max\{2i-(i+1): 1\leq i\leq n\}=n-1,$$ as desired.
	\end{proof}
	
	In view of this result, we pose the following question. In Section 2, we will give a positive answer to this question, when $G$ is a tree or a connected unicyclic graph whose unique cycle has length $3$ or $4$.
	
	\begin{Question}\label{Quest:reg-R-I-c-G}
		Let $G$ be a finite simple graph on $[n]$. Is it true that
		$$
		\reg\,\mathcal{R}(I_c(G))\le n\ ?
		$$
	\end{Question}
	
	\section{Koszulness of $\mathcal{R}(I_c(G))$}\label{sec2}
	
	In this section, we ask when $\mathcal{R}(I_c(G))$ has a quadratic Gr\"obner basis and when it is Koszul. First, we give necessary conditions for $\mathcal{R}(I_c(G))$ to be Koszul.
	
	\begin{Theorem}\label{Koszul} Let $G$ be a finite simple graph. If  $\mathcal{R}(I_c(G))$ is Koszul, then $c(G)=1$ and $G$ satisfies the following conditions.

            \begin{enumerate}
\item[{\rm (i)}]
Any even cycle $C$ of $G$ of length $\geq 6$
  has either an even-chord or three odd-chords
$e, e', e''$ such that $e$ and $e'$ cross in $C$.
\item[{\rm (ii)}]
If $C_1$ and $C_2$ are minimal odd cycles of $G$ with
exactly one common vertex, then there exists an edge
$\{i, j\} \not\in E(C_1) \cup E(C_2)$ with
$i \in V(C_1)$, $j \in V(C_2)$.
\item[{\rm (iii)}]
If $C_1$ and $C_2$ are minimal odd cycles with
$V(C_1) \cap V(C_2) = \emptyset$,
then there exist at least two bridges
between $C_1$ and $C_2$.
\end{enumerate}
\end{Theorem}

	\begin{proof}
		Since $\mathcal{R}(I_c(G))$ is Koszul, by \cite[Corollary 3.6]{Blum}, $I_c(G)^k$ has linear resolution for all $k\geq 1$. Then using \cite[Corollary 3.2]{FM1} we have $c(G)=1$. 

        Now, we show that $G$ satisfies the conditions (i) to (iii). To this aim, by \cite[Theorem 1.2]{HO2} (see also \cite[Theorem 5.14]{HHBook2}), it is enough to show that the defining ideal $L$ of the edge ring $K[G]$ is generated by quadratic binomials. Since $\mathcal{R}(I_c(G))$ is Koszul, its defining ideal $J$ is generated by quadratic binomials. Now, consider a binomial relation $z_A-z_B\in K[z_1,\ldots,z_m]$ of the edge ring $K[G]$. Here, $z_F=\prod_{i\in F} z_i$ for a subset $F\subset [m]$. Then $y_A-y_B\in J$. So there are quadratic binomials $f_{A_1},\ldots,f_{A_r} \in J$, where each $f_{A_i}$ is a quadratic binomial in $K[y_1,\ldots,y_m]$ such that 
        $y_A-y_B= \sum_{i=1}^r u_i f_{A_i}$ and $u_1,\ldots u_r$ are monomials in $K[y_1,\ldots,y_m]$. For any $i$, let $f_{A_i}=y_{a_i}y_{b_i}-y_{c_i}y_{d_i}$. We set $f'_{A_i}=z_{a_i}z_{b_i}-z_{c_i}z_{d_i}$. Clearly, $f'_{A_1},\ldots,f'_{A_r}\in L$. Moreover, $z_A-z_B= \sum_{i=1}^r v_i f'_{A_i}$, where $v_i=\prod_{y_j\mid u_i} z_j$ for each $i$. This shows that $L$ is indeed generated by quadratic binomials.    
		\end{proof}

        Next, we provide large families of graphs for which $\mathcal{R}(I_c(G))$ is Koszul. 
		
        \begin{Theorem}\label{tree}
 If $G$ is a tree, then the defining ideal $J$ of $\mathcal{R}(I_c(G))$ has a quadratic Gröbner basis with respect to some monomial order. In particular $\mathcal{R}(I_c(G))$ is Koszul.
	\end{Theorem}

    \begin{proof}
         Let $G$ be a tree with $n$ vertices. We label the vertices of $G$ such that for each $1\leq r\leq n-1$, the vertex $r$ is a leaf of $G_r=G[r,r+1,\ldots,n]$. Here, by $G[r,r+1,\ldots,n]$ we mean the induced subgraph of $G$ on the vertex set $\{r,r+1,\ldots,n\}$. 
		Consider the lex order $<$ on $T$ induced by $x_1>\dots>x_n>y_1>\dots>y_m$. We prove that $J$ has a quadratic Gröbner basis with respect to this order. 
		
		Consider a minimal monomial generator $uy_{i_1}\cdots y_{i_k}\in \ini_<(J)$, where $u\in S$ is a monomial. Let $h=uy_{i_1}\cdots y_{i_k}-vy_{j_1}\cdots y_{j_k}\in J$ be a primitive binomial with $\ini_<(h)=uy_{i_1}\cdots y_{i_k}$, such that $v\in S$ is a monomial with $\gcd(u,v)=1$ and $u>_\lex v$. 
		Then the relation $h\in J$ gives the relation $h'=uz_{j_1}\cdots z_{j_k}-vz_{i_1}\cdots z_{i_k}\in J'$.
		Since $h$ is a primitive binomial in $J$, $h'$ is a primitive binomial in $J'$. 
		So by \cite[Corollary 10.1.5]{HHBook}, the relation $h'$ corresponds to a primitive even closed walk in $G^*$, say $W$. Since $G$ has no cycle, $W$ contains the vertex $n+1$, which means that $\deg(u)\geq 1$. On the other hand, the labeling on $G$ is so that $G_r$ is connected for all $r$. Hence, by \cite[Theorem 3.1]{FM1} and its proof, $\deg(u)\leq 1$. Hence $u=x_p$ for some $p$ and $v=x_q$ for some $q>p$. Moreover, $W$ is of the form $n+1,p,\ell_1,\ldots,\ell_{2k-1},q,n+1$ with $k\geq 1$. Since $G$ has no cycles, $\ell_1,\ldots,\ell_{2k-1},p,q\in V(G)$ are distinct vertices. We set $\ell_0=p$ and $\ell_{2k}=q$. Then after relabeling the edges, we may assume that $e_{j_t}=\{\ell_{2t-1},\ell_{2t}\}$ and $e_{i_t}=\{\ell_{2t-2},\ell_{2t-1}\}$, for any $1\leq t\leq k$.
		
		We claim that $\ell_1>p$. Suppose on the contrary that $\ell_1<p$.
		By assumption, $G_p=G[p,p+1,\ldots,q]$ is connected. So there is a path $L$ from $p$ to $q$ in $G_p$. Since $\ell_1<p$, the path $L$ does not contain the vertex $\ell_1$. Hence, $L$ is different from the path $p,\ell_1,\ldots,\ell_{2k-1},q$. This means that there are at least two paths from $p$ to $q$ in $G$, which contradicts to the fact that $G$ is a tree.  
		Thus, $\ell_1>p$, as was claimed. Then $\ell_1$ and $q$ are distinct vertices of the connected graph $G_{p+1}=G[p+1,\ldots,n]$. So by connectedness of $G_{p+1}$, there exists a vertex $s>p$ such that $\{\ell_1,s\}\in E(G)$. Then $\{\ell_1,s\}=e_t$ for some $t$ and $x_pz_t-x_sz_{i_1}\in J'$.  Therefore, $g=x_py_{i_1}-x_sy_t\in J$. Since $p<s$, we have $\ini_<(g)=x_py_{i_1}$. Clearly, $\ini_<(g)$ divides $\ini_<(h)$ and by the minimality of $\ini_<(h)$ we obtain $\ini_<(h)=\ini_<(g)=x_py_{i_1}$. Thus $\ini_<(J)$ is generated by quadratic monomials of the form $x_iy_j$. Hence, by \cite[Theorem 2.28]{HHBook2},  $\mathcal{R}(I_c(G))$ is Koszul. 
	\end{proof}
	
	Now, let $G$ be a connected unicyclic graph with unique cycle $C$ of length $d$. In order to study the defining ideal $J$ of $\mathcal{R}(I_c(G))$, in the next two theorems we consider the following labeling on $V(G)$. For any $1\leq i\leq n-d$, let $i$ be a leaf of the graph $G_i=G[i,i+1,\ldots,n]$. Moreover, we label the vertices of $C$ by $n-d+1,n-d+2,\ldots,n$ such that $\{i,i+1\}\in E(G)$ for $n-d+1\leq i\leq n-1$. We consider the lex order on the polynomial ring $T=S[y_1,\dots,y_m]$ induced by the order $x_1>\cdots>x_n>y_1>\cdots>y_m$ and denote this order by $<'$. 
	
	\begin{Theorem}\label{unicylic}
		Let $G$ be a unicyclic graph with a cycle of length $d$. Then 
		\begin{enumerate}
			\item [(a)] The  ideal $J$  has a quadratic Gröbner basis with respect to $<'$ if and only if $c(G)=1$ and $d\in\{3,4\}$.
			\item [(b)] If $c(G)=1$ and $d\in\{3,4\}$, then  $\mathcal{R}(I_c(G))$ is Koszul.
		\end{enumerate}
	\end{Theorem}
	
	\begin{proof} 
		(a) Let $G$ be a unicyclic graph with a $3$-cycle $C$ and $c(G)=1$. Since removing isolated vertices does not change $J$, we may assume that $G$ is connected. 
		Let    $uy_{i_1}\cdots y_{i_k}\in \ini_{<'}(J)$ be a minimal monomial generator, where $u\in S$ is a monomial, and let $h=uy_{i_1}\cdots y_{i_k}-vy_{j_1}\cdots y_{j_k}\in J$ be a primitive binomial with $\ini_{<'}(h)=uy_{i_1}\cdots y_{i_k}$, such that $v\in S$ is a monomial with $\gcd(u,v)=1$ and $u>_\lex v$.  Then $h$ corresponds to a primitive even closed walk in $G^*$, say $W$.
		The labeling on $G$ described before the statement of the theorem, implies that $G_r=G[r,r+1,\ldots,n]$ is connected for all $r$. Hence, by \cite[Theorem 3.1]{FM1} and its proof,  $\deg(u)\leq 1$. Since $G$ has no even closed walks, $W$ contains the vertex $n+1$, which means that $\deg(u)=1$.
		Hence $u=x_p$ for some $p$ and $v=x_q$ for some $q>p$. So $W$ is of the form $$n+1,p=\ell_0,\ell_1,\ldots,\ell_{2k-1},\ell_{2k}=q,n+1,$$ where $\ell_1,\ldots,\ell_{2k-1},p,q\in V(G)$ and $k\geq 1$. For any $1\leq t\leq k$, after relabeling the edges we have $e_{j_t}=\{\ell_{2t-1},\ell_{2t}\}$ and $e_{i_t}=\{\ell_{2t-2},\ell_{2t-1}\}$.
		We show that $\{\ell_1,s\}\in E(G)$ for some $s>p$. Once we show this, the same argument as in the proof of Theorem~\ref{tree} implies that a quadratic monomial of the form $x_iy_j\in \ini_{<'}(J)$ divides $\ini_{<'}(h)$, as desired. If $p\in V(C)$, then the inequality $p<q$, and the labeling on $G$ imply that $q\in V(C)$. From this together with the assumptions that $W$ is primitive and $G$ is has a unique cycle of length $3$, we conclude that $\{\ell_0,\ell_1,\ldots,\ell_{2k-1},\ell_{2k}\}\subset V(C)$. Thus $k=1$ and $W$ is a $4$-cylce $W: n+1,p,\ell_1,q,n+1$. So taking $s=q$, we have $\{\ell_1,s\}\in E(G)$ with $s>p$.
		Now, consider the case that $p\notin V(C)$. First, we show that $\ell_1>p$. 
		By contradiction assume that $\ell_1<p$. Then $\ell_1\notin V(C)$. Since $\ell_1$ is a leaf of $G[\ell_1,\ell_1+1,\ldots,n]$, by $\ell_1<p$ and $\{p,\ell_1\},\{\ell_1,\ell_2\}\in E(G)$ and that $\ell_2\neq p$, we obtain $\ell_2<\ell_1$. Hence, $\ell_2\notin V(C)$.   Similar arguments imply the inequalities  $q>p>\ell_1>\cdots>\ell_{2k-1}$.
		Then $\ell_{2k-2}$ and $q$ are distinct vertices adjacent to $\ell_{2k-1}$ in $G[\ell_{2k-1},\ldots,n]$, which contradicts to $\ell_{2k-1}$ being a leaf of $G[\ell_{2k-1},\ldots,n]$. Thus $\ell_1>p$, as desired. Next, we show that $\ell_2>p$. 
		Suppose on the contrary that $\ell_2<p$. This implies that $\ell_2\notin V(C)$,  $\ell_2<\ell_1$, and that $\ell_1$ is a vertex of $G[\ell_2,\ldots,n]$ which is adjacent to $\ell_2$. If $\ell_2<\ell_3$, then $\ell_3$ is adjacent to $\ell_2$ in $G[\ell_2,\ldots,n]$, as well, which contradicts to $\ell_2$ being a leaf of $G[\ell_2,\ldots,n]$. Hence, $\ell_3<\ell_2$. Similar arguments show that $\ell_{2k-1}<\ell_{2k-2}<\cdots <\ell_2<p<q$. Thus $\ell_{2k-2}$ and $q$ are adjacent to $\ell_{2k-1}$ in $G[\ell_{2k-1},\ldots,n]$, which contradicts to $\ell_{2k-1}$ being a leaf of $G[\ell_{2k-1},\ldots,n]$. Thus $\ell_2>p$.    
		Since $\{\ell_1,\ell_2\}\in E(G)$,  the desired vertex $s$ is $s=\ell_2$. 
		The proof is complete in the case of $d=3$.   
		
		\smallskip
		
		Now, let $G$ be a connected unicyclic graph with a $4$-cycle $C$. 
		For a minimal monomial generator $uy_{i_1}\cdots y_{i_k}$ of $\ini_{<'}(J)$, if $\deg(u)=1$, then the same argument as in the case of the $3$-cycle shows that 
		$uy_{i_1}\cdots y_{i_k}=x_iy_j$ for some $i$ and $j$. Now let $\deg(u)=0$. Then the primitive binomial $h=y_{i_1}\cdots y_{i_k}-y_{j_1}\cdots y_{j_k}$ corresponds to a primitive even closed walk in $G$. Since the only primitive even closed walk in $G$ is the $4$-cycle $C$, $h$ is a quadratic binomial. Hence, $\ini_{<'}(J)$ is generated by quadratic monomials.
		
		Conversely, assume that $J$ has a quadratic Gröbner basis with respect to  $<'$. Then $\mathcal{R}(I_c(G))$ is Koszul. So by Theorem~\ref{Koszul}, we have $c(G)=1$ and $G$ has no induced even cycle of length $\geq 6$. By contradiction assume that  $d\geq 5$. Since $C$ is an induced cycle of $G$, we obtain that $d$ is odd.
        So $d=2k+1$ for some $k\geq 2$, and $C$ is the cycle on the vertices $n-2k,n-2k+1,\ldots,n$. For each $0\leq \ell\leq 2k-1$, let $i_\ell$ be the integer with $\{n-2k+\ell,n-2k+\ell+1\}=e_{i_{\ell}}$. Moreover, we let $\{n,n-2k\}=e_{i_{2k}}$. Then 
		$x_{n-1}z_{i_1}z_{i_3}\cdots z_{i_{2k-3}}z_{i_{2k}}-x_{n}z_{i_0}z_{i_2}\cdots z_{i_{2k-2}}\in J'$. Hence, 
		$$g=x_{n-1}y_{i_0}y_{i_2}\cdots y_{i_{2k-2}}-x_{n}y_{i_1}y_{i_3}\cdots y_{i_{2k-3}}y_{i_{2k}}\in J,$$ and $\ini_{<'}(g)=x_{n-1}y_{i_0}y_{i_2}\cdots y_{i_{2k-2}}$. Since $J$ has a quadratic Gröbner basis with respect to $<'$, a monomial $w\in\ini_{<'}(J)$ of degree two divides $x_{n-1}y_{i_0}y_{i_2}\cdots y_{i_{2k-2}}$. 
		From $d=2k+1$, we know that $G$ has no even cycle. Thus $w=x_{n-1}y_{i_t}$ for some $t\in\{0,2,\ldots,2k-2\}$. Let $g_0=x_{n-1}y_{i_t}-x_sy_j\in J$ be a relation with $\ini_{<'}(g_0)=x_{n-1}y_{i_t}$. Then we have $s=n$.
		The relation $g_0=x_{n-1}y_{i_t}-x_ny_j$ corresponds to $x_{n-1}z_j-x_nz_{i_t}\in J'$ and hence, to a $4$-cycle of the form $n+1,n-1,\lambda,n,n+1$ in $G^*$, where 
		$e_{i_t}=\{n-1,\lambda\}$ and $e_j=\{n,\lambda\}$. Notice that by the labeling on $V(G)$ we have $\{n-1,n\}\in E(G)$. Therefore, $n-1,\lambda,n$ form a $3$-cycle in $G$, which contradicts to the fact that $G$ is a unicyclic graph with a cycle of length $d\geq 5$.
		
		\smallskip
		(b) follows from (a) and  \cite[Theorem 2.28]{HHBook2}.

	\end{proof}
	
	Using Theorem \ref{tree} and Theorem \ref{unicylic}, we are able to give a positive answer to Question \ref{Quest:reg-R-I-c-G} for trees and connected unicyclic graphs with the unique cycle of length $3$ or $4$ in the following corollary.  
	Recall that a {\em matching} $M$ in a graph $G$ is a set of pairwise disjoint edges of $G$. The {\em matching number} of $G$ is the largest size of a matching of $G$ and is denoted by $\mat(G)$. 
	
	\begin{Corollary}\label{bound-reg}
		Let $G$ be a tree or a connected unicyclic graph with the unique cycle of length $d\in\{3,4\}$.  Then
		$$
		\reg\,\mathcal{R}(I_c(G))\le |V(G)|. 
		$$ 
	\end{Corollary}
	
	\begin{proof}
		By Theorem \ref{tree}, Theorem \ref{unicylic}(a) and their proofs, there exists a monomial order $<$ on $T$ such that $\ini_<(J)$ is generated by squarefree monomials of the forms $x_iy_j$ and $y_ry_s$. Therefore, $\ini_<(J)$ is the edge ideal of a  graph $H$ on the vertex set $V(H)=\{x_1,\ldots,x_n\}\cup\{y_1,\ldots,y_m\}$, where $n=|V(G)|$ and $m=|E(G)|$. Since $\ini_<(J)$ is squarefree, by \cite[Corollary 2.7]{CV}, we have $$\reg\,\mathcal{R}(I_c(G))=\reg\,T/J=\reg\,T/\ini_<(J)=\reg\,T/I(H).$$  
		By \cite[Theorem 6.7]{HV}, we have $\reg\,T/I(H)\leq\mat(H)$. Since $G$ is either a tree or unicyclic, we have $n-1\leq m\leq n$. Thus $|V(H)|\leq 2n$. Therefore, $\mat(H)\leq n$. This shows that $\reg\,\mathcal{R}(I_c(G))=\reg\,T/I(H)\leq n$.
	\end{proof}
	
	\section{Normality of $\mathcal{R}(I_c(G))$}\label{sec3}

    In this section, we put together known results on the normality of the Rees algebras and the toric rings of $I(G)$ and $I_c(G)$.  Moreover, we give an independent proof for the equality $\ell(I(G))=\ell(I_c(G))=n-b(G)$.
	
	Recall that a graph $G$ is said to satisfy the \textit{odd cycle condition}, if for any two  odd cycles $C_1$ and $C_2$ of $G$, either $C_1$ and $C_2$ have a common vertex or there exist $i\in V(C_1)$ and $j\in V(C_2)$ such that $\{i,j\}\in E(G)$.\smallskip
    
   Let $\m=(x_1,\dots,x_n)$. For an ideal $I\subset S$, the \textit{fiber cone} $\mathcal{R}(I)/\m\mathcal{R}(I)$ of $I$ is denoted by $\mathcal{F}(I)$. Combining results from \cite{HO},\cite{SVV}, \cite{RV}, and \cite{ALS} we obtain
	\begin{Theorem}\label{Thm:Fiber-IcG}
		For a finite simple graph $G$, the following conditions are equivalent.
		\begin{enumerate}
                \item[\textup{(a)}] $\mathcal{R}(I(G))$ is normal.
                \item[\textup{(b)}] $\mathcal{F}(I(G))$ is normal.
			\item[\textup{(c)}] $\mathcal{R}(I_c(G))$ is normal.
			\item[\textup{(d)}] $\mathcal{F}(I_c(G))$ is normal.
			\item[\textup{(e)}] $G$ satisfies the odd cycle condition.
		\end{enumerate}
	\end{Theorem}
	\begin{proof}
		Let $V(G)=[n]$. Since $I(G)$ and $I_c(G)$ are equigenerated ideals, then $\mathcal{F}(I(G))$ is isomorphic to the edge ring $K[G]=K[x_ix_j:\{i,j\}\in E(G)]$ and similarly $\mathcal{F}(I_c(G))\cong K[{\bf x}_{[n]}/(x_ix_j):\{i,j\}\in E(G)]$. Combining \cite[Corollary 5.8.10]{SVV}  (see also \cite[Corollary 2.3]{HO}) with \cite[Corollary 10.5.6]{RV}, the equivalences (a) $\Leftrightarrow$ (b) $\Leftrightarrow$ (e) follow. Next, by \cite[Theorem 3.1]{ALS}, we have $\mathcal{F}(I(G))\cong\mathcal{F}(I_c(G))$. So, the equivalence (b) $\Leftrightarrow$ (d) follows. Finally, the equivalence (a) $\Leftrightarrow$ (c) follows from \cite[Corollary 14.6.36]{RV}.
	\end{proof}

	For a finite simple graph $G$, we denote by $b(G)$ the number of bipartite connected components of $G$. An isolated vertex of $G$ is regarded as a bipartite connected component of $G$.
	
	Recall that the \textit{analytic spread} of an ideal $I\subset S$ is the Krull dimension of the fiber cone $\mathcal{F}(I)=\mathcal{R}(I)/\m\mathcal{R}(I)$, and it is denoted by $\ell(I)$. If $I\subset S$ is an equigenerated monomial ideal and $\mathcal{G}(I)=\{u_1,\dots,u_m\}$, then $\mathcal{F}(I)\cong K[u_1,\dots,u_m]$ is a toric ring. Let $M=(m_{ij})$ be the $m\times n$ matrix whose $i$th row is the exponent vector of the monomial $u_i$. By \cite[Proposition 3.1]{HHBook2}, we have $\ell(I)=\rank(M)$.
	
	As a consequence of this discussion, \cite[Lemma 10.2.6]{RV} and the isomorphism $\mathcal{F}(I_c(G))\cong\mathcal{F}(I(G))$, we obtain immediately that
	
	\begin{Corollary}\label{Cor:ell-I(G)}
		Let $G$ be a finite simple graph on $n\ge3$ vertices. Then
		$$
		\ell(I_c(G))=\ell(I(G))=n-b(G).
		$$
	\end{Corollary}
	
	For the sake of completeness, we provide an independent proof of this result using elementary linear algebra. First, we need the following lemma.
	\begin{Lemma}\label{Lem:matrix}
		Let
		\begin{enumerate}
			\item[(i)] $B=(b_{ij})\in\mathbb{R}^{n\times m}$ be a real matrix such that the sum of the entries of each column is a fixed value $\sum_{i=1}^n b_{ij}=b>0$.
			\item[(ii)] $A=(a_{ij})\in\mathbb{R}^{n\times m}$ be a real matrix such that $a_{ij}=a_{ij'}$ for all $i,j,j'$ and such that the sum of the entries of each column is a fixed value $\sum_{i=1}^n a_{ij}=a>b$.
		\end{enumerate}
		Then $\rank(A-B)=\rank(B)$.
	\end{Lemma}
	\begin{proof}
		By the Rank-Nullity Theorem we have $\rank(A-B)=m-\dim\Ker(A-B)$ and $\rank(B)=m-\dim\Ker(B)$. So, it is enough to show that $\Ker(A-B)=\Ker(B)$.
        
        Let ${\bf y}\in\Ker(A-B)$, then $(A-B){\bf y}={\bf 0}$. This means that
		\begin{equation}\label{eq:system}
			\sum_{j=1}^m(a_{ij}-b_{ij})y_j=0,\quad\textup{for all}\ i=1,\dots,n.
		\end{equation}
		Summing over $i$, we obtain
		$$
		0=\sum_{i=1}^n\sum_{j=1}^m(a_{ij}-b_{ij})y_j=\sum_{j=1}^m(\sum_{i=1}^na_{ij}-\sum_{i=1}^nb_{ij})y_j=\sum_{j=1}^m(a-b)y_j=(a-b)(\sum_{j=1}^my_j).
		$$
		Since $a>b$, then $a-b>0$ and so $y_1+\dots+y_m=0$. Combining this fact with equation (\ref{eq:system}) and the assumption in (ii) that $a_{ij}=a_{ij'}$ for all $i,j,j'$, we see that
		$$
		0=-\sum_{j=1}^m(a_{ij}-b_{ij})y_j=-a_{i1}(\sum_{j=1}^my_j)+\sum_{j=1}^mb_{ij}y_j=\sum_{j=1}^mb_{ij}y_j,
		$$
		for all $i=1,\dots,n$. Hence ${\bf y}\in\Ker(B)$.
		
		Conversely, let ${\bf y}\in\Ker(B)$. Then
		\begin{equation}\label{eq:system2}
			\sum_{j=1}^mb_{ij}y_j=0,\quad\textup{for all}\ i=1,\dots,n.
		\end{equation}
		Summing these equations over $i$, we obtain that $b(y_1+\dots+y_m)=0$. Since $b>0$, we see that $y_1+\dots+y_m=0$. Using this fact, the equation (\ref{eq:system2}), and the assumption in (ii) that $a_{ij}=a_{ij'}$ for all $i,j,j'$, we obtain that
		$$
		\sum_{j=1}^m(a_{ij}-b_{ij})y_j=a_{i1}(\sum_{j=1}^my_j)-(\sum_{j=1}^mb_{ij}y_j)=0,
		$$
		for all $i=1,\dots,n$. Hence ${\bf y}\in\Ker(A-B)$.
	\end{proof}
	
	We are now ready to prove Corollary \ref{Cor:ell-I(G)}.
	\begin{proof}[Proof of Corollary \ref{Cor:ell-I(G)}]
		Let $V(G)=[n]$, $E(G)=\{e_1,\dots,e_m\}$, and let $B=(b_{ij})$ be the \textit{incidence matrix} of $G$. That is, the $m\times n$-matrix defined by
		$$
		b_{ij}\ =\ \begin{cases}
			1&\textup{if}\ j\in e_i,\\
			0&\textup{if}\ j\notin e_i.
		\end{cases}
		$$
        
		Using that $\mathcal{F}(I(G))\cong K[x_ix_j:\{i,j\}\in E(G)]$, by \cite[Proposition 3.1]{HHBook2}, we have $\ell(I(G))=\rank(B)$. Let $A$ be the $m\times n$-matrix whose all entries are 1's. Similarly, we have $\ell(I_c(G))=\rank(A-B)$ because $\mathcal{F}(I_c(G))\cong K[{\bf x}_{[n]}/(x_ix_j):\{i,j\}\in E(G)]$. The conditions (i)-(ii) in Lemma \ref{Lem:matrix} are satisfied for $A^\top$ and $B^\top$, where $C^\top$ is the transpose of a matrix $C$. Hence $$\rank(A-B)=\rank((A-B)^\top)=\rank(A^\top-B^\top)=\rank(B^\top)=\rank(B),$$
        and so $\ell(I(G))=\ell(I_c(G))$.
		
		Finally, it remains to show that $\ell(I(G))=\rank(B)=n-b(G)$. This is well-known (see \cite[Lemma 10.2.6]{RV}). We sketch a short argument. Let $G=G_1\sqcup\cdots\sqcup G_t\sqcup G_{t+1}$, where each $G_i$, $1\leq i\leq t$, is a connected component of $G$ with at least two vertices, and $G_{t+1}$ consists of the isolated vertices of $G$. Then, up to relabeling, $B$ is a diagonal block matrix
		$$
		B\ =\ \left(\begin{array}{cccc}
			B_1&&&{\Huge{\bf 0}}\\
			&B_2&&\\
			&&\ddots&\\
			{\Huge{\bf 0}}&&&B_t
		\end{array}\right)
		$$
		where each $B_i$ is the incidence matrix of $G_i$. Then $\rank(B)=\sum_{i=1}^t\rank(B_i)$. Since $b(G)=(\sum_{i=1}^t b(G_i))+|V(G_{t+1})|$ and $n=|V(G)|=\sum_{i=1}^{t+1}|V(G_{i})|$, we may assume that $G$ is connected. Hence, by the Rank-Nullity Theorem, it is enough to show that $\dim\Ker(B)=1$ if $G$ is bipartite, and $\dim\Ker(B)=0$ otherwise. Notice that the system $B{\bf y}=(0,\dots,0)$ can be rewritten as the system of equations
		\begin{equation}\label{eq:system-y}
			y_p+y_q\ =\ 0,\quad \textup{for}\ e=\{p,q\}\in E(G).
		\end{equation}
		
		\textbf{Case 1.} Assume that $G$ is a connected bipartite graph with vertex bipartition $V(G)=V_1\sqcup V_2$. We claim that $\dim\Ker(B)=1$. To this end, let $v,v'\in V_1$ be distinct. Let ${\bf y}=(y_1,\dots,y_n)^\top\in\Ker(B)$. Since $G$ is connected, we can find a path $v=v_0,v_1,\dots,v_{r-1},v_r=v'$ in $G$ connecting $v$ with $v'$. Since $G$ is bipartite and $v_0=v\in V_1$, then $v_1\in V_2$. For the same reason, $v_2\in V_1$. Therefore, $v_i\in V_1$ if $i$ is even and $v_i\in V_2$ if $i$ is odd. Since $v_r=v'\in V_1$, we see that $r$ is even. Using the system (\ref{eq:system-y}), we see that $y_v=y_{v'}$. By symmetry, $y_{v}=y_{v'}$ for all $v,v'\in V_2$. Up to relabeling, we may assume that $V_1=\{1,\dots,t\}$ and $V_2=\{t+1,\dots,n\}$. Let $e\in E(G)$. Since $G$ is bipartite, $e=\{i,j\}$ with $1\le i\le t$ and $t+1\le j\le n$. Our discussion shows that $y_1=\dots=y_t$ and $y_{t+1}=\dots=y_{n}$ and $y_i+y_j=0$. It follows that $y_j=-y_i$ for all $i\in V_1$ and $j\in V_2$. Hence
		$$
		\Ker(B)=\{(a,\dots,a,-a,\dots,-a)^\top\in\mathbb{R}^{1\times n}:\ a\in\mathbb{R}\},
		$$
		and consequently $\dim\Ker(B)=b(G)=1$.\smallskip
		
		\textbf{Case 2.} Suppose that $G$ is a connected non-bipartite graph. By \cite[Lemma 9.1.1]{HHBook}, $G$ contains an odd cycle $C$. Say $E(C)=\{\{1,2\},\{2,3\},\dots,\{2s,2s+1\},\{2s+1,1\}\}$ with $s\ge1$. Let ${\bf y}=(y_1,\dots,y_n)^\top\in\Ker(B)$. Then, (\ref{eq:system-y}) implies that
		$$
		y_{i}+y_{i+1}\ =\ 0,\quad \textup{for}\ i=1,\dots,2s+1,
		$$
		where $y_{2s+2}=y_1$. From these equations, we have $y_{i}=y_{i+2}$ for all $i=1,\dots,2s$. Since $C$ is an odd cycle, $y_1=y_2=\cdots=y_{2s+1}$. Hence $2y_1=0$ and so $y_1=\dots=y_{2s+1}=0$. If $V(C)=V(G)$, then $\Ker(B)$ is the null space and so $\dim\Ker(B)=0$. Otherwise, let $v\in V(G)\setminus V(C)$. Since $G$ is connected, we can find a path in $G$, say $v=v_0,v_1,\dots,v_{r-1},v_r=1$, 
        with $\{v_i,v_{i+1}\}\in E(G)$ for $i=0,\dots,r-1$, connecting $v$ to $1$. Let $r$ be even. Using the system (\ref{eq:system-y}), we see that $y_{v_0}=y_{v_2}=\dots=y_{v_r}=y_1$. Otherwise, let $r$ be odd, we have $y_v=y_{v_0}=y_{v_{r-1}}$. Since $\{v_{r-1},v_r\}\in E(G)$ and $\{v_r,2\}=\{1,2\}\in E(G)$, the system (\ref{eq:system-y}) implies that $y_{v_{r-1}}=y_2$. But $y_2=y_1$ and so $y_{v}=y_{v_{r-1}}=y_1$. So, $y_i=y_1=0$ for all $i\in V(G)$. Hence $\Ker(B)$ is the null space and so $\dim\Ker(B)=b(G)=0$, as claimed.
	\end{proof}

	\section{The limit depth $\lim\limits_{k\rightarrow\infty}\depth S/I_c(G)^k$}\label{sec4}
	
	Recall that, by \cite{Br}, the limit $\lim_{k\rightarrow\infty}\depth S/I^k$ exists for any ideal $I\subset S$. That is, $\depth S/I^k=\depth S/I^{k+1}$ for all $k\gg0$. The least integer $k_0>0$ for which $\depth S/I^k=\depth S/I^{k_0}$ for all $k\ge k_0$, is called the \textit{index of depth stability} of $I$ and is denoted by $\dstab\,I$.
	
	The main aim of this section is to prove the following theorem. 
	
	\begin{Theorem}\label{Thm:limitdepth-Ic(G)}
		Let $G$ be a finite simple graph with $n$ vertices. Then
		$$
		\lim_{k\rightarrow\infty}\depth S/I_c(G)^k\ =\ b(G),
		$$
		and $\dstab\,I_c(G)\le n-c(G)-1$.
	\end{Theorem}
	
	The proof of this result requires some preparation.\smallskip
	
	Recall that a monomial ideal $I\subset S$ has \textit{linear quotients} if there exists an order $u_1,\dots,u_m$ on the minimal generating set $\mathcal{G}(I)$ of $I$ such that $(u_1,\dots,u_{i-1}):(u_i)$ is generated by variables, for all $i=2,\dots,m$. We put
	$$
	\set_{I}(u_j)\ =\ \{i:\ x_i\in(u_1,\dots,u_{j-1}):(u_j)\},
	$$
	for $j=2,\dots,m$, and $\set_I(u_1)=\emptyset$.
	
	\begin{Lemma}\label{Lem:I^k-lin-quot}
		Let $I\subset S$ be an equigenerated monomial ideal. Suppose that $I^k$ has linear quotients with respect to the lexicographic monomial order $>_{\lex}$ induced by $x_1>\dots>x_n$, for all $k\ge1$. Then,
		\begin{enumerate}
			\item[\textup{(a)}] $\set_{I^k}(u)\subset[n-1]$, for all $u\in\mathcal{G}(I^k)$ and all $k\ge1$.
			\item[\textup{(b)}] $\set_{I^k}(u)\cup\set_{I^\ell}(v)\subset\set_{I^{k+\ell}}(uv)$, for all $u\in\mathcal{G}(I^k)$ and $v\in\mathcal{G}(I^\ell)$.
			\item[\textup{(c)}] $\depth S/I^k=0$, if and only if, $\set_{I^k}(u)=[n-1]$, for some $u\in\mathcal{G}(I^k)$.
			\item[\textup{(d)}] Suppose that $\lim_{k\rightarrow\infty}\depth S/I^k=n-|\bigcup_{u\in\mathcal{G}(I)}\set_{I}(u)|-1$. Then
			$$
			\dstab\,I\ \le\ \min\big\{|A|:A\subset\mathcal{G}(I),\ \bigcup_{u\in\mathcal{G}(I)}\set_I(u)=\bigcup_{v\in A}\set_I(v)\big\}\ \le\ \big|\!\bigcup_{u\in\mathcal{G}(I)}\set_{I}(u)\big|.
			$$
		\end{enumerate}
	\end{Lemma}
	\begin{proof}
		(a) Let $u\in\mathcal{G}(I^k)$. Then $i\in\set_{I^k}(u)$, if and only if, $u'=x_i(u/x_j)\in\mathcal{G}(I^k)$ and $u'>_{\lex}u$, for some $j$. Therefore, $x_i>x_j$, i.e., $i<j$. Hence, $\set_{I^k}(u)\subset[n-1]$.
        \smallskip
		
		(b) Let $u\in\mathcal{G}(I^k)$ and $v\in\mathcal{G}(I^\ell)$. If $i\in\set_{I^k}(u)$, then $u'=x_i(u/x_j)\in\mathcal{G}(I^k)$ for some $j>i$. Hence $u'v>_{\lex}uv$ and $u'v\in\mathcal{G}(I^{k+\ell})$. This shows that $i\in\set_{I^{k+\ell}}(uv)$. Similarly, $\set_{I^\ell}(v)\subset\set_{I^{k+\ell}}(uv)$. 
		 \smallskip
         
		(c) By \cite[Corollary 8.2.2]{HHBook}, the Auslander-Buchsbaum formula and the assumption that $I^k$ has linear quotients, we have $\depth S/I^k=\min_{u\in\mathcal{G}(I^k)}\{n-|\set_{I^k}(u)|-1\}$. Combining this fact with (a), we see that $\depth S/I^k=0$, if and only if, there exists $u\in\mathcal{G}(I^k)$ such that $\set_{I^k}(u)=[n-1]$.

         \smallskip
         
		(d) Let $s$ be the minimum cardinality of a subset $A=\{u_{1},\dots,u_{s}\}$ of $\mathcal{G}(I)$ such that $\bigcup_{u\in\mathcal{G}(I)}\set_I(u)=\bigcup_{i=1}^{s}\set_I(u_i)$. Put $v=u_1\cdots u_s$. By (b), $\set_{I^s}(v)$ contains $\bigcup_{u\in\mathcal{G}(I)}\set_I(u)$. So, by \cite[Corollary 8.2.2]{HHBook}, $\depth S/I^s\le n-|\bigcup_{u\in\mathcal{G}(I)}\set_{I}(u)|-1$. Since $I$ has linear powers, by \cite[Proposition 10.3.4]{HHBook} the function $k\mapsto\depth S/I^k$ is non-increasing. Using this, the previous inequality and the assumption, we have
		\begin{align*}
			n-\big|\!\bigcup_{u\in\mathcal{G}(I)}\set_{I}(u)\big|-1\ &\ge\ \depth S/I^s\ge\depth S/I^{s+1}\ge\depth S/I^{s+2}\ge\cdots\\[-6pt]
			&\ge\ \lim_{k\rightarrow\infty}\depth S/I^k\ =\ n-\big|\!\bigcup_{u\in\mathcal{G}(I)}\set_{I}(u)\big|-1.
		\end{align*}
		Hence $\depth S/I^k=n-|\bigcup_{u\in\mathcal{G}(I)}\set_{I}(u)|-1$ for all $k\ge s$, and so
		$$
		\dstab\,I\ \le\ s\ \le\ \big|\!\bigcup_{u\in\mathcal{G}(I)}\set_{I}(u)\big|,
		$$
		as desired.
	\end{proof}
	
	Now, we treat the case of connected bipartite graphs.
	\begin{Proposition}\label{Prop:dstab-conn}
		Let $G$ be a connected bipartite graph on $n\geq 3$ vertices. Then $\dstab\,I_c(G)\le n-2$, and
        $$
        \lim_{k\rightarrow\infty}\depth\,S/I_c(G)^k=1.
        $$
	\end{Proposition}
	\begin{proof}
		Let $V(G)=[n]$. Since $G$ is bipartite, by \cite[Lemma 9.1.1]{HHBook}it does not contain induced odd cycles. Hence $G$ satisfies the odd cycle condition. Theorem \ref{Thm:Fiber-IcG} implies that $\mathcal{R}(I_c(G))$ is normal, and hence it is Cohen-Macaulay by \cite[Theorem B.6.2]{HHBook}. This combined with \cite[Proposition 10.3.2]{HHBook} and Corollary \ref{Cor:ell-I(G)}, implies that
        $$
        \lim_{k\rightarrow\infty}\depth S/I_c(G)^k=n-\ell(I_c(G))=b(G)=1.
        $$
        Next, we may assume that $G_r=G[r,r+1,\dots,n]$ is connected for all $r=1,\dots,n$, see the proof of \cite[Theorem 3.1(b)]{FM}. By \cite[Theorem 3.1(b)]{FM} and \cite[Remark 3.3]{FM}, $I_c(G)^k$ has linear quotients with respect to the lexicographic order $>_{\lex}$ induced by $x_1>\dots>x_n$, for all $k\ge1$. Proceeding by induction on $n\ge3$, we will show that 
    \begin{equation}\label{UnionSets}
        \bigcup_{u\in\mathcal{G}(I_c(G))}\set_{I_c(G)}(u)\ =\ [n-2].
    \end{equation}
        Since $\lim_{k\rightarrow\infty}\depth S/I_c(G)^k=1$, having (\ref{UnionSets}) together with Lemma \ref{Lem:I^k-lin-quot}(d) will imply that $\dstab\,I_c(G)\le n-2$, as desired.
		
		For the base case $n=3$, we have that $G=P_3$ is a path on three vertices, $I_c(G)=(x_1,x_3)$ and so $\bigcup_{u\in\mathcal{G}(I_c(G))}\set_{I_c(G)}(u)=\set_{I_c(G)}(x_3)=\{1\}=[n-2]$.
		
		Now, let $n>3$. Notice that $H=G\setminus\{1\}$ is again connected and bipartite on $n-1$ vertices. Therefore by induction $\bigcup_{u\in\mathcal{G}(I_c(H))}\set_{I_c(H)}(u)=\{2,3,\dots,n-2\}$. Notice that for any $u\in\mathcal{G}(I_c(H))$, we have $x_1u\in\mathcal{G}(I_c(G))$ and $\set_{I_c(G)}(x_1u)$ contains $\set_{I_c(H)}(u)$. Therefore, $\bigcup_{u\in\mathcal{G}(I_c(G))}\set_{I_c(G)}(u)$ contains $\{2,3,\dots,n-2\}$. Since $G$ is connected, we have $\{1,p\}\in E(G)$ for some $p>1$. Since $G_2=G[2,\dots,n]$ is connected on $n-1\ge2$ vertices and $p\in V(G_2)$ we have $\{p,q\}\in E(G)$ for some $q>1$. Notice that $u={\bf x}_{[n]}/(x_px_q)>_{\lex}{\bf x}_{[n]}/(x_1x_p)=v$, both $u,v\in\mathcal{G}(I_c(G))$, and $u:v=\lcm(u,v)/v=x_1$. Hence $1\in\set_{I_c(G)}(u)$. Therefore $[n-2]\subset\bigcup_{u\in\mathcal{G}(I_c(G))}\set_{I_c(G)}(u)$. If the inclusion was not an equality, then Lemma \ref{Lem:I^k-lin-quot}(a) would imply that $\bigcup_{u\in\mathcal{G}(I_c(G))}\set_{I_c(G)}(u)=[n-1]$. Then, Lemma \ref{Lem:I^k-lin-quot}(b) implies that for all $k\gg0$ large enough, there exists $v_k\in\mathcal{G}(I_c(G)^k)$ such that $\set_{I_c(G)^k}(v_k)=[n-1]$. Lemma \ref{Lem:I^k-lin-quot}(c) then implies that $\lim_{k\rightarrow\infty}\depth S/I_c(G)^k=0$ against the fact that this limit is equal to $b(G)=1$. Hence $\bigcup_{u\in\mathcal{G}(I_c(G))}\set_{I_c(G)}(u)=[n-2]$.
	\end{proof}

    The following lemma is needed to treat the case of connected non-bipartite graphs.

    \begin{Lemma}\label{Lem:removal}
            Let $G$ be a connected graph having a cycle $C$ such that $|V(G)|>|V(C)|$. Then, there exists $v\in V(G)\setminus V(C)$ such that $G\setminus\{v\}$ is connected.
        \end{Lemma}
        \begin{proof}
        Let $T$ be a spanning tree of $G$. Then $T$ has at least four vertices. Any leaf $w$ of $T$ is such that $G\setminus\{w\}$ is connected. We distinguish two cases.\smallskip

        \textbf{Case 1.} Suppose there exists a leaf $w\in V(T)$ such that $w\notin V(C)$. Then $G\setminus\{w\}$ is connected and $w\in V(G)\setminus V(C)$.\smallskip

        \textbf{Case 2.} Suppose that all leaves of $T$ belong to $V(C)$. Pick any $w\in V(T)\setminus V(C)$. We claim that $G\setminus\{w\}$ is connected. Let $u,v\in V(G)\setminus\{w\}$ be distinct vertices. Then $u,v\in V(T)$ and since $T$ is a tree, there is a path in $T$ from $u$ to $v$. Let $P:\ v_0,v_1,\dots,v_{r-1},v_r$ be a maximal path in $T$ which contains $u$ and $v$, with $\{v_i,v_{i+1}\}\in E(T)\subset E(G)$ for $i=0,\dots,r-1$. Then by the maximality of $P$, we have that $v_0,v_r$ are leaves of $T$. Let $0\le i<j\le r$ be such that $u=v_i$ and $v=v_j$. If $w\ne v_h$ for all $i+1\le h\le j-1$, then $u$ and $v$ are connected in $G\setminus\{w\}$ via the path $P$. Suppose that $w=v_h$ for some $i+1\le h\le j-1$. All the leaves of $T$ belong to the cycle $C$. Hence $v_0,v_r\in V(C)$ and this shows that $u$ and $v$ are connected by a path in $G\setminus\{w\}$. We conclude that $G\setminus\{w\}$ is connected.
        \end{proof}
	
	\begin{Proposition}\label{Prop:nonbip}
		Let $G$ be a connected non-bipartite graph on $n\ge3$ vertices. Then $\dstab\,I_c(G)\le n-2$, and $$\lim_{k\rightarrow\infty}\depth S/I_c(G)^k=0.$$
	\end{Proposition}
	\begin{proof}
		Let $G$ be a connected non-bipartite graph. By \cite[Lemma 9.1.1]{HHBook}, $G$ contains an induced odd cycle $C$. We prove the statement proceeding by induction on the integer $t=|V(G)|-|V(C)|\ge0$.
        
        For the base case, let $t=0$. Then $V(G)=V(C)$. Let $C=C_{2s+1}$ with $s\ge1$. Then we may assume that $V(G)=[2s+1]$ and $E(C)=\{\{1,2\},\dots,\{2s,2s+1\},\{2s+1,1\}\}$. We claim that $\m=(x_1,\dots,x_{2s+1})\in\Ass\,I_c(G)^s$. If $s=1$, then $I_c(G)=(x_1,x_2,x_3)=\m$ and so $\depth S/I_c(G)^k=0$ for all $k\ge 1$, as desired. Now, let $s\ge2$, and put $u=(x_1\cdots x_{2s+1})^{s-1}$. Notice that
		$$
		x_1u\ =\ \prod_{i=1}^s\big(\frac{x_1x_2\cdots x_{2s+1}}{x_{2i}x_{2i+1}}\big)\in I_c(G)^s.
		$$
		By symmetry, we have $x_iu\in I_c(G)^s$ for all $1\le i\le 2s+1$. Hence $\m\subset I_c(G)^s:(u)$. On the other hand, $u\notin I_c(G)^s$ because $\deg(u)=(2s+1)(s-1)<(2s-1)s$ and $I_c(G)^s$ is generated in degree $(2s-1)s$. Hence $I_c(G)^s:(u)=\m$. This shows that $\depth S/I_c(G)^s=0$. By \cite[Theorem 4.1]{FM1}, the depth function $k\mapsto\depth S/I_c(G)^k$ is non-increasing. That is $\depth S/I_c(G)^k\ge\depth S/I_c(G)^{k+1}$ for all $k\ge1$. Hence, $\depth S/I_c(G)^k=0$ for all $k\ge s$, and in particular for all $k\ge n-2=2s-1$.\smallskip
		
		Now, suppose that $t\ge1$. By Lemma \ref{Lem:removal}, there exists a vertex $j\in V(G)\setminus V(C)$ such that $G\setminus\{j\}$ is connected. Up to relabeling, $j=1$. Then, we can determine an order of the vertices $1,2,\dots,n$ of $G$ such that $G\setminus\{1,2,\dots,i\}$ is connected for all $i$ (see \cite[Proof of Theorem 3.1(a)]{FM}). Let $H=G\setminus\{1\}$. By \cite[Remark 3.3]{FM}, $I_c(G)^k$ and $I_c(H)^k$ have linear quotients with respect to the lexicographic monomial order induced by $x_1>\dots>x_{n}$, for all $k\ge1$. Since $C$ is contained in $H$ and $H$ is connected, by induction we have $\depth\,S/I_c(H)^k=0$ for all $k\ge|V(H)|-2=n-3$. Using Lemma \ref{Lem:I^k-lin-quot}(c), this means that for all $k\ge n-3$, there exists a monomial $v_k\in\mathcal{G}(I_c(H)^{k})$ such that $\set_{I_c(H)^{k}}(v_k)=\{2,3,\dots,n-1\}$. Notice that $w_k=x_1^{k}v_k\in\mathcal{G}(I_c(G)^{k})$ and clearly $\set_{I_c(G)^{k}}(w_k)$ contains $\{2,3,\dots,n-1\}$. We have $\{1,p\}\in E(G)$ for some $p>1$. Since $H$ is connected, we also have $\{p,q\}\in E(G)$ for some $q>1$ with $p\ne q$. Notice that ${\bf x}_{[n]}/(x_px_q)>_{\lex}{\bf x}_{[n]}/(x_1x_q)$ and setting $u={\bf x}_{[n]}/(x_1x_q)$ we have $1\in\set_{I_c(G)}(u)$ because $x_1(u/x_p)={\bf x}_{[n]}/(x_px_q)>_{\lex}u$. Now, using Lemma \ref{Lem:I^k-lin-quot}(a)-(b), we see that $\set_{I_c(G)^{k+1}}(uw_k)=[n-1]$ for all $k\ge n-3$. Lemma \ref{Lem:I^k-lin-quot}(c) shows that $\depth\,S/I_c(G)^k=0$ for all $k\ge n-2$. Hence $\dstab\,I_c(G)\le n-2$.
	\end{proof}
	
	Now, we are in the position to prove Theorem \ref{Thm:limitdepth-Ic(G)}.
	\begin{proof}[Proof of Theorem \ref{Thm:limitdepth-Ic(G)}]
		Let $j\in V(G)$ be an isolated vertex of $G$ and $H=G\setminus\{j\}$. Then $I_c(G)^k=x_j^kI_c(H)^k$. Suppose that the statements hold for $H$. Then,
		\begin{align*}
			\lim_{k\rightarrow\infty}\depth S/I_c(G)^k\ &=\ \lim_{k\rightarrow\infty}\depth S/I_c(H)^k\\
			&=\ \lim_{k\rightarrow\infty}\depth K[x_i:i\in V(H)]/I_c(H)^k+1\\[2pt]
			&=\ b(H)+1\ =\ b(G),
		\end{align*}
		where we used that $b(G)=b(H)+1$ (the component $\{v\}$ consisting of an isolated vertex is bipartite). Notice moreover, that $\dstab\,I_c(G)=\dstab\,I_c(H)$. Since also $c(G)=c(H)$, we have $|V(G)|-c(G)-1>|V(H)|-c(H)-1$. So we may assume that $G$ does not contain isolated vertices.
		
		Now, we proceed by induction on $c(G)$. If $c(G)=1$, then $G$ is connected. In this case, the assertion holds by Propositions \ref{Prop:dstab-conn} and \ref{Prop:nonbip}.
		
		Next, suppose now $c(G)>1$, and write $G=G_1\sqcup G_2$ with $G_2$ a connected graph. Identifying the variables of $S$ with the vertices of $G$, we may assume that $V(G_1)=\{x_1,\dots,x_n\}$ and $V(G_2)=\{y_1,\dots,y_m\}$. Let $S_1=K[x_1,\dots,x_n]$ and $S_2=K[y_1,\dots,y_m]$. Then $S=S_1\otimes_K S_2$. Moreover, we put
        $$
            I_1=({\bf x}_{[n]}/(x_ix_j):\{x_i,x_j\}\in E(G_1)),\quad
            I_2=({\bf y}_{[m]}/(y_iy_j):\{y_i,y_j\}\in E(G_2)).
        $$
        Since $c(G_1),c(G_2)<c(G)$, by induction we have
		\begin{align}
			\label{eq:depth1}\depth S_1/I_1^k\ &=\ b(G_1),\quad \textup{for all}\ k\ge n-c(G),\\
			\label{eq:depth2}\depth S_2/I_2^k\ &=\ b(G_2),\quad \textup{for all}\ k\ge m-2,
		\end{align}
		where we used that $c(G_1)=c(G)-1$ and $c(G_2)=1$.
		
		Let $I=I_c(G)$. The proof of \cite[Theorem 4.1]{FM1} shows that
		\begin{align}
			\label{eq:depth3}\depth\frac{S}{I^k}=\min&\left\{\substack{\displaystyle\!\depth S_1/I_1^k\!+m,\depth S_2/I_2^k\!+n,\depth S_2/I_2^{k-1}\!+n-1,\\[5pt]\displaystyle\min_{0<h<k}\{\depth S_1/I_1^{k-h}+\depth S_2/I_2^h\}}\right\},
		\end{align}
		for all $k\ge1$. Recall that by \cite[Theorem 4.1]{FM1}, each depth function appearing in the above formula is non-increasing. That is,
		$$
		\depth S_i/I_i^k\ge\depth S_i/I_i^{k+1},\quad\textup{for all}\ k\ge1,\ \textup{and}\ i=1,2.
		$$
		Combining these inequalities with the formulas (\ref{eq:depth1}), (\ref{eq:depth2}) and (\ref{eq:depth3}), it follows that
		$$
		\depth S/I^k\ge b(G_1)+b(G_2)=b(G),
		$$ for all $k\ge1$.
		On the other hand, let $k\ge n+m-c(G)-1$, and $h=k-(n-c(G))$. Then, $k-h=n-c(G)$, $h\ge m-1>m-2$, $0<h<k$. So the formulas (\ref{eq:depth1}), (\ref{eq:depth2}) and (\ref{eq:depth3}) imply that $$\depth S/I^k\le\depth S_1/I_1^{k-h}+\depth S_2/I_2^h=b(G_1)+b(G_2)=b(G),$$
	for all $k\ge n+m-c(G)-1=|V(G)|-c(G)-1$. Hence, inequality holds for all $k\ge|V(G)|-c(G)-1$. This shows that $\lim_{k\rightarrow\infty}\depth S/I_c(G)^k=b(G)$ and that $\dstab\,I_c(G)\le|V(G)|-c(G)-1$.
	\end{proof}
	
	The bound for $\dstab\,I_c(G)$ given in Theorem \ref{Thm:limitdepth-Ic(G)} is sharp. Indeed, we have
	\begin{Proposition}\label{Prop:Pn}
		Let $G=P_n$ be the path graph on $n\ge3$ vertices. Then
            \begin{equation}\label{eq:depthP_n}
               \depth\,S/I_c(G)^k\ =\ \begin{cases}
                n-k-1&\textup{for}\ 1\le k\le n-3,\\
                \hfil1&\textup{for}\ k\ge n-2.
                \end{cases}
            \end{equation}
            In particular, $\textup{dstab}(I_c(P_n))=n-2$.
	\end{Proposition}
	\begin{proof}
		Since the order $1,\dots,n$ has obviously the property that $G_r=G[r,\dots,n]$ is connected for all $r=1,\dots,n$, by \cite[Theorem 3.1 and Remark 3.3]{FM}, $I_c(G)^k$ has linear quotients for all $k\ge1$, with respect to the lexicographic order $>_{\lex}$ induced by $x_1>\dots>x_n$. By Theorem \ref{Thm:limitdepth-Ic(G)}, $\depth S/I_c(G)^k=1$ for all $k\ge n-2$. So we may assume that $1\le k\le n-3$.  We prove by induction on $n$ that $\depth\,S/I_c(G)^k=n-k-1$. For the base case $n=3$ there is nothing to prove. Now, let $n>3$ and set $H=G\setminus\{1\}$. Then $H$ is a path on $n-1$ vertices and so, by induction on $n$, we have $\depth S/I_c(H)^k=n-k-2$ for $1\le k\le n-4$ and $\depth S/I_c(H)^k=1$ for $k\ge n-3$. We can write $I_c(G)=x_1I_c(H)+(x_3\cdots x_n)$. We put $v=x_3\cdots x_n$. Then $I_c(G)^k=\sum_{\ell=0}^k x_1^{k-\ell}v^\ell I_c(H)^{k-\ell}$, for all $k\ge1$. We claim that
		\begin{equation}\label{eq:BS}
			J_h\ =\ (\sum_{\ell=0}^{h-1}x_1^{k-\ell}v^\ell I_c(H)^{k-\ell})+x_1^{k-h}v^h I_c(H)^{k-h}
		\end{equation}
		is a Betti splitting for all $h=1,\dots,k$.
        
        To this end, it is clear that $\mathcal{G}(J_h)$ is the disjoint union of $\mathcal{G}(\sum_{\ell=0}^{h-1}x_1^{k-\ell}v^\ell I_c(H)^{k-\ell})$ and $\mathcal{G}(x_1^{k-h}v^h I_c(H)^{k-h})$, because the monomials in these two sets all have degree $(n-2)k$, but they have different $x_1$-degree. Since each power of $I_c(G)$ and $I_c(H)$ has linear quotients with respect to the order $>_{\lex}$, we see that both the ideals $J_{h-1}=\sum_{\ell=0}^{h-1}x_1^{k-\ell}v^\ell I_c(H)^{k-\ell}$ and $x_1^{k-h}v^h I_c(H)^{k-h}$ have linear quotients, and therefore linear resolution. By \cite[Corollary 2.4]{FHT}, it follows that (\ref{eq:BS}) is indeed a Betti splitting.
		
		Next, we compute the intersection
		\begin{align*}
			J_{h-1}\cap(x_1^{k-h}v^h I_c(H)^{k-h})&=\sum_{\ell=0}^{h-1}[(x_1^{k-\ell}v^\ell I_c(H)^{k-\ell})\cap(x_1^{k-h}v^h I_c(H)^{k-h})]\\
			&=\sum_{\ell=0}^{h-1}x_1^{k-\ell}v^h I_c(H)^{k-h}= (x_1^k,x_1^{k-1},\dots,x_1^{k-h+1})v^hI_c(H)^{k-h}\\[2pt]
			&=x_1^{k-h+1}v^hI_c(H)^{k-h}.
		\end{align*}
		In the above equalities, we used that $v^hI_c(H)^{k-h}\subset v^{h-1}I_c(H)^{k-(h-1)}\subset\cdots\subset I_c(H)^k$. This follows because $v=x_3\cdots x_n=x_3(x_2x_3\cdots x_n)/(x_2x_3)\in I_c(H)$.
		
		Since (\ref{eq:BS}) is a Betti splitting, and $J_k=I_c(G)^k$, the above computations show that
		\begin{align*}
			\depth S/I_c(G)^k\ &=\ \min\{\depth S/J_{k-1},\,\depth S/(v^k),\,\depth S/(x_1v^k)-1\}\\
			&=\ \min\{\depth S/J_{k-1},\,n-2\}.
		\end{align*}
		Now, let $R=K[x_2,\dots,x_n]$. Recall that $\depth S/(fJ)=\depth S/J$ for any ideal $J\subset S$ and any $f\in S$. Iterating the above computations to $J_{k-1},\dots,J_1$, and using that $x_1$ does not divide any minimal monomial generator of $I_c(H)$, we then see that
		\begin{align*}
			\,\depth S/I_c(G)^k\ =\ \min\bigg\{\substack{\displaystyle\depth R/I_c(H)^k+1,\,n-2,\\[4pt]
				\displaystyle\min_{0<h<k}\{\depth R/I_c(H)^{h}\}}\bigg\}.
		\end{align*}
		Since $\depth R/I_c(H)^{k}=n-k-2$ for $1\le k\le n-4$ and $\depth S/I_c(H)^k=1$ for $k\ge n-3$, the above formula implies that (\ref{eq:depthP_n}) holds.
	\end{proof}
    
    By \cite[Proposition 10.3.2]{HHBook} and Corollary \ref{Cor:ell-I(G)}, if $\mathcal{R}(I_c(G))$ is Cohen-Macaulay, then $$
	\lim_{k\rightarrow\infty}\depth S/I_c(G)^k\ =\ |V(G)|-\ell(I_c(G))\ =\ b(G).
	$$
    In view of this fact, Theorem \ref{Thm:limitdepth-Ic(G)}, and several experimental evidence, we are tempted to conclude the paper by posing the following conjecture.
	\begin{Conj}\label{Conj:R(Ic(G))-CM}
		Let $G$ be a finite simple graph. Then $\mathcal{R}(I_c(G))$ is Cohen-Macaulay.
	\end{Conj}
	
	This conjecture holds true for any graph $G$ satisfying the odd cycle condition, by combining Theorem \ref{Thm:Fiber-IcG} with \cite[Theorem B.6.2]{HHBook}.\medskip
	
	\textbf{Acknowledgment.} We thank Rafael H. Villarreal for his useful comments that helped us to greatly improve the quality of the manuscript. A. Ficarra was supported by the Grant JDC2023-051705-I funded by
	MICIU/AEI/10.13039/501100011033 and by the FSE+ and also by INDAM (Istituto Nazionale Di Alta Matematica).

\bigskip
	
	\begin{thebibliography}{99}
		
		\bibitem{ALS} K. Ansaldi, K. Lin, Y. Shen, \textit{Generalized Newton complementary duals of monomial ideals}, J. Algebra Appl. {\bf 20} (2021), no. 2, Paper No. 2150021.

        \bibitem{Blum} S. Blum, \textit{Subalgebras of bigraded Koszul algebras}, J. Algebra {\bf 242} (2001), no. 2, 795--809. 
		
		\bibitem{Br} M. Brodmann, \textit{The asymptotic nature of the analytic spread}, Math. Proc. Cambridge Philos. Soc., {\bf 86} (1979), 35--39.
		
		\bibitem{CV}  A. Conca, M. Varbaro, \textit{Square-free Gröbner degenerations},  Invent. Math. {\bf 221} (2020), no. 3, 713--730.
		
		\bibitem{FShort} A. Ficarra, \textit{A new proof of the Herzog-Hibi-Zheng theorem}, Comm in Algebra, to appear, 2025, https://doi.org/10.1080/00927872.2025.2537276.
		
		\bibitem{FM} A. Ficarra, S. Moradi, \textit{Stanley-Reisner ideals with linear powers}, (2025), preprint \url{https://arxiv.org/abs/2508.10354}.
		
		\bibitem{FM1} A. Ficarra, S. Moradi, \textit{Complementary edge ideals}, (2025), preprint \url{https://arxiv.org/abs/2508.10870}.
		
		\bibitem{FHT} C.~A.~Francisco, H.~T.~Hà, A.~Van Tuyl, \textit{Splittings of monomial ideals}, Proc. Amer. Math. Soc. {\bf 137}  (2009), no. 10, 3271-3282.

        
		\bibitem{HV}  H.~T.~Hà, A.~Van Tuyl, \textit{Monomial ideals, edge ideals of hypergraphs, and their graded Betti numbers}, J. Algebraic Combin. {\bf 27} (2008), no. 2, 215--245.
		
		\bibitem{HHBook} J.~Herzog, T.~Hibi, \emph{Monomial ideals}, Graduate texts in Mathematics {\bf 260}, Springer, 2011.
		
		\bibitem{HHBook2} J.~Herzog, T.~Hibi, H. Ohsugi, \emph{Binomial ideals},  Cham: Springer {\bf 279}, Springer, 2018.
		
		\bibitem{HHZ} J.~Herzog, T.~Hibi, X.~Zheng, \textit{Monomial ideals whose powers have a linear resolution}, Math. Scand. {\bf 95} (2004), no. 1, 23--32.

        \bibitem{HO} T. Hibi, H. Ohsugi, \textit{Normal polytopes arising from finite graphs}, J. Algebra {\bf 207}, (1998), 409–426.

        \bibitem{HO2} T. Hibi, H. Ohsugi, \textit{Toric ideals generated by quadratic binomials}, J. Algebra  {\bf 218},
(1999), 509–527. 
		
		\bibitem{HQS} T. Hibi, A. A. Qureshi, S. Saeedi Madani, \textit{Complementary edge ideals}, (2025), preprint \url{https://arxiv.org/abs/2508.09837}.

       \bibitem{SVV} A. Simis, W. V. Vasconcelos, R. H. Villarreal, \textit{On the ideal theory of graphs}, J. Algebra {\bf 167}, (1994), 389–416.
       
     \bibitem{V} R. H. Villarreal, \textit{Rees algebras of edge ideals}, Comm. Algebra {\bf 23}, (1995), no. 9, 3513--3524.

     \bibitem{RV2} R. H. Villarreal, \textit{Rees algebras and polyhedral cones of ideals of vertex covers of perfect graphs}, J. Algebraic Combin. {\bf 27}, (2008), no. 3, 293--305.
		
		\bibitem{RV} R. H. Villarreal, \emph{Monomial Algebras}, 2nd Edition (Chapman and Hall/CRC, 2018).
		
	\end{thebibliography}
\end{document}